\documentclass[12 pt]{amsart}

\usepackage{bigints}
\usepackage{caption}
\usepackage{tikz-cd}
\usepackage[all]{xy}
\usepackage{enumerate}
\usepackage{amssymb}
\usepackage{float}
\usepackage{setspace}
\usepackage{comment}
\usepackage{xcolor}
\usepackage[utf8]{inputenc}
\usepackage[T1]{fontenc}
\usepackage{epsfig}
\usepackage{amsmath}
\usepackage{amsthm}
\usepackage{psfrag}
\usepackage{amsmath, amsthm, amssymb}

\newcommand{\be}{\begin{equation}}
\newcommand{\ee}{\end{equation}}

\newtheorem{theorem}{Theorem}[section]

\newtheorem{proposition}[theorem]{Proposition}
\newtheorem{lemma}[theorem]{Lemma}

\newtheorem{definition}[theorem]{Definition}

\newtheorem{remark}[theorem]{Remark}

\newtheorem{example}[theorem]{Example}

\def\v{\upsilon}

\newfont{\graf}{eufm10}

\newcommand{\bdm}{\begin{displaymath}}
\newcommand{\edm}{\end{displaymath}}

\def\haken{\mathbin{\hbox to 6pt{%
                 \vrule height0.4pt width5pt depth0pt
                 \kern-.4pt
                 \vrule height6pt width0.4pt depth0pt\hss}}}

\usepackage{amssymb}
\usepackage{amsmath}
\usepackage{amscd}
\usepackage{graphicx}
\usepackage{latexsym}

\calclayout\evensidemargin .6in

\author{Eyüp Yalçınkaya}
\date{\today}
\address{\noindent University of Rochester
\newline \indent Department of Mathematics, NY, USA}
\email{eyup.yalcinkaya@rochester.edu}

\subjclass[2000]{53C38,53D05,53D15,57R17}
\keywords{(Almost) symplectic structures, $Spin(7)$ structures, Mirror Duality}

\begin{document}

\title{the 4-planes of Spin(7) manifolds}



\subjclass[2000]{53C38,53D05,53D15,57R17}
\keywords{(Almost) symplectic structures, $Spin(7)$ structures, Mirror Duality}


\maketitle

\begin{abstract}
This paper investigates the geometric structures and properties of 8-dimensional manifolds with Spin(7)-holonomy. We focus on the characterization and implications of 4-planes within these manifolds, which are endowed with an almost symplectic structure compatible with the Spin(7)-structure. We provide a detailed analysis of the differential forms defining these structures, including the Cayley 4-form and its stabilization under the Spin(7) group actions. Furthermore, we explore the concept of mirror duality within the context of exceptional holonomy, elucidating the relationship between dual structures and their topological constraints. By examining the conditions under which mirror duality arises, we enhance the understanding of the interplay between Spin(7)-structures and symplectic geometry. This endeavor contributes to our deeper understanding of the intricate symmetries and topological properties of these remarkable manifolds.

\end{abstract}


\section{Introduction}

String theory, an overarching framework seeking to reconcile general relativity and quantum mechanics, has long captivated the imagination of physicists and mathematicians alike. This theory postulates that the fundamental building blocks of the universe are not point-like particles but rather tiny, vibrating strings. To achieve its ambitious goals, string theory necessitates a backdrop of spacetime that extends beyond the conventional dimensions of our observable universe. In particular, theories with aspirations of uniting these two pillars of modern physics often invoke dimensions beyond the familiar three spatial dimensions and one-time dimensions. Within this context, 8-dimensional manifolds have emerged as a particularly captivating field of study, due to their rich mathematical properties and their potential relevance to string theory.

Given an orientable 8-dimensional manifold, one remarkable subgroup of $SO(8)$ is the one associated with the exceptional Lie group Spin(7). These manifolds are of particular interest because they naturally admit the structure of a Spin(7) manifold, a property that makes them exceptionally valuable in the study of string theory. The connection between Spin(7) manifolds and string theory is a profound one, as the geometry of these manifolds has the potential to illuminate the intricate symmetries and topological features that underlie fundamental physical processes.

In this paper, we delve into the intriguing relationship between string theory and Spin(7) manifolds, focusing on the role that a 4-frame field can play within this context. By investigating the symmetries and topological conditions associated with Spin(7) manifolds, we aim to shed light on the profound connections between geometry and theoretical physics.


Manifolds bearing these distinctive holonomy groups have earned their prominence as integral components in the realm of M-theory compactifications. These are the enigmatic, intricately folded dimensions that lie concealed at every point in the fabric of spacetime. Alongside their more explored counterparts like the 6-dimensional Calabi-Yau manifolds, and the 7-dimensional $G_2$ manifolds, the 8-dimensional $Spin(7)$ manifolds constitute a class of great interest.

Despite the extensive research devoted to Calabi-Yau manifolds, the geometric intricacies of $G_2$ and $Spin(7)$ manifolds have remained largely uncharted territory. This research is aimed at unraveling the mysteries of mirror duality within the realm of Riemannian 8-manifolds with spin structures. Then it represents a critical step in the broader quest to comprehend the profound relationships between geometry and theoretical physics, further contributing to the ongoing discourse surrounding the unification of fundamental forces, particularly within the intricate landscape of string theory.

\vspace{.1in}

\noindent In particular, we know the following theorem;

\vspace{.1in}

\noindent {\bf Theorem:}  {(\cite{SE}) \em All the odd-dimensional Stiefel-Whitney classes of a smooth, closed, connected, orientable 8-manifold with spin structure vanish.}

\vspace{.1in}

Note that a manifold $M$ with $Spin(7)$-the structure is orientable and spin. The theorem above implies that the obstructions for the existence of almost symplectic (and hence almost complex) structures on a manifold with full $Spin(7)$ holonomy vanish as well. There are inclusions between the groups 

\vspace{.1in}

$$SU(2) \longrightarrow SU(3)\longrightarrow G_2 \longrightarrow Spin(7), $$


\noindent and 

$$ SU(2) \times SU(2) \longrightarrow Sp(2) \longrightarrow SU(4) \longrightarrow Spin(7). $$

\vspace{.1in}

These are the only connected Lie subgroups of $Spin(7)$ which can be holonomy groups of Riemannian metrics on 8-manifolds. Hence the theorem above also holds for 8-manifolds with reduced holonomy groups.






\vspace{.1in} 
 Existence of 4-vector fields on given manifolds depends on its topology. The necessary conditions for Spin(7) manifolds are discussed the following sections.

Let $M$ be a Spin(7) manifold with a 4-vector field $\{u,v,y,w\}$. Almost complex structures induced by triple product define mirror pairs. 
Finally, we will prove the following;

\vspace{.1in} 
  \noindent {\bf Theorem:} {\em Let $M$ be a Spin(7) manifold. $\{u,v,y,w\}$ span $X$ as a 4-vector field. Then, 
  non-degenerate two forms $\omega_{uv}$ and $ \omega_{yw}$  generates Calabi-Yau manifolds $(N_1,J_{uv},\omega_{uv},Re\ \Omega_{uv},Im\ \Omega_{uv})$ and $(N_2,J_{yw},\omega_{yw},Re\ \Omega_{yw},Im\ \Omega_{yw})$ mirror dual of each other induced from $\Phi$. }
 
 \noindent Note that $SU(3)$-structure can be induced by Spin(7)-structure by foliation. By means of this reduction, we construct mirror Calabi-Yau submanifolds in Spin(7) manifolds.
\section{Background}

Let $(M,g)$ be a Riemannian manifold. A closed differential $p$-form $\phi$ on $M$ with $p \geq 1$ is termed a calibration if it satisfies
\[
\phi|_\xi \leq \text{vol}|_\xi
\]
for any oriented $p$-dimensional tangent plane $\xi$ in $T_xM$ at any point $x \in M$. When a Riemannian manifold $(M, g)$ has such a differential form $\phi$, it is known as a calibrated manifold. A $p$-plane $\xi$ is referred to as a $\phi$-plane or calibrated plane if
\[
\phi|_\xi = \text{vol}|_\xi,
\]
and the collection of all $\phi$-planes at a point $x$ is called the $\phi$-Grassmannian, denoted by $G(\phi_x)$. This is a subset of the oriented Grassmannian $G^+_p(M)$, which consists of all oriented $p$-planes in $T_xM$. Additionally, the collection of all $p$-planes on $M$ forms a subbundle of the oriented Grassmannian bundle $\tilde{G}_p(M)$ over $M$.

Each calibration $\phi$ defines a specific geometry for submanifolds. An oriented $p$-dimensional submanifold $X^p$ within a calibrated manifold $(M, \phi)$ is called a calibrated submanifold or a $\phi$-submanifold if each tangent space $T_xX \subset T_xM$ belongs to $G(\phi_x)$. A key result about these submanifolds, known as the fundamental theorem of calibrated geometry, is that all $\phi$-submanifolds minimize volume absolutely within their homology class.

There are numerous interesting examples of calibrations with rich geometric structures, especially in manifolds with special holonomy. Manifolds such as Kähler, Calabi--Yau, hyper-Kähler, $G_2$, or $\text{Spin}(7)$ have one or more natural calibrations.  For further examples of calibrations, one can refer to the seminal work by Harvey and Lawson \cite{Lawson} or the comprehensive book by Joyce \cite{Joyce}, which primarily explores manifolds with special holonomy. In this paper, we focus on the flat $\text{Spin}(7)$-manifold $\mathbb{R}^8$ with the Cayley calibration.

\subsection{$Spin(7$)-structures}

In this section we review the basics of $Spin(7)$ geometry. More on the subject can be found in  \cite{Fernandez}, \cite{Joyce}, \cite{Lawson} and \cite{EI}.

\vspace{.1in}

Let  $(x^1,..., x^8)$ be coordinates on $\mathbb{R}^8$. The standard Cayley 4-form on $\mathbb{R}^8$  can be written as 

\begin{align*}
\Phi_0&=dx^{1234}+dx^{1256}+dx^{1278}+dx^{1357}-dx^{1368}-dx^{1458}-dx^{1467}\\
&-dx^{2358}-dx^{2367}-dx^{2457}+dx^{2468}+dx^{3456}+dx^{3478}+dx^{5678}
\end{align*}

\vspace{.1in}

\noindent where $dx^{ijkl}=dx^i\wedge dx^j\wedge dx^k \wedge dx^l$.

\vspace{.1in}

The subgroup of $GL(8, \mathbb{R})$ that preserves $\Phi_0$ is the group $Spin(7)$. It is a 21-dimensional compact, connected and simply-connected Lie group that preserves the orientation on $\mathbb{R}^8$ and the Euclidean metric $g_0$. 
\vspace{.1in}

\begin{definition}
		A differential 4-form $\Phi$ on an oriented 8-manifold $M$ is called admissible if it can be identified with $\Phi_0$ through an oriented isomorphism between $T_pM$ and $\mathbb{R}^8$ for each point $p\in M$. 
	\end{definition}
 
 \vspace{.1in}




\begin{definition} Let $\mathcal{A}(M)$ denotes the space of admissible 4-forms on $M$. A $Spin(7)$-structure on an 8-dimensional manifold $M$ is an admissible 4-form $\Phi \in \mathcal{A}(M)$. If $M$ admits such structure, $(M, \Phi)$ is called a manifold with $Spin(7)$-structure. 
\end{definition}


\noindent Each 8-manifold with a $Spin(7)$-structure $\Phi$ is canonically equipped with a metric $g$. Hence, we can think of a $Spin(7)$-structure on $M$ as a pair $(\Phi, g)$ such that for all $p \in M$ there is an isomorphism between $T_pM$ and $\mathbb{R}^8$ which identifies $(\Phi_p, g_p)$ with $(\Phi_0, g_0)$. 

\vspace{.1in}

\noindent The existence of a $Spin(7)$-structure on an 8-dimensional manifold $M$ is equivalent to a reduction of the structure group of the tangent bundle of $M$ from $SO(8)$ to its subgroup $Spin(7)$. The following result gives the necessary and sufficient conditions so that the 8-manifold admits $Spin(7)$ structure.



\vspace{.1in}

\begin{theorem} (\cite{Lawson}, \cite{Fernandez})
	Let $M$ be a differentiable 8-manifold. $M$ admits a $Spin(7)$-structure if and only if $w_1(M)=w_2(M)=0$ and for appropriate choice of orientation on $M$ we have that 
	$$ p_1(M)^2-4p_2(M)\pm 8\chi(M)=0.$$
\end{theorem}

\vspace{.1in}


\noindent Furthermore, if $\nabla \Phi=0$, where $\nabla$ is the Riemannian connection of $g$, then $\text{Hol}(M) \subseteq Spin(7)$, and $M$ is called a $Spin(7)$-manifold. All  $Spin(7)$ manifolds are Ricci flat.

\subsection{Decomposition of $\bigwedge^*(M)$ into irreducible Spin(7)-representations} \label{decom}

Let $(M,g,\Phi)$ be a Spin(7) manifold. The action of Spin(7) on the tangent space gives an action of $Spin(7)$ on the spaces of differential forms, 
$\Lambda^k(M)$, and so the exterior algebra splits orthogonally
into components, where $\Lambda^k_l$ corresponds to an irreducible
representation of $Spin(7)$ of dimension $l$:

 \vspace{.1in}

$$\Lambda^1(M)=\Lambda^1_8, \quad \Lambda^2(M) = \Lambda^2_7\oplus
\Lambda^2_{21}, \quad
\Lambda^3(M)=\Lambda^3_8\oplus\Lambda^3_{48}, $$ $$
\Lambda^4(M)=\Lambda^4_+(M)\oplus \Lambda^4_-(M), \quad
\Lambda^4_+(M)=\Lambda^4_1\oplus\Lambda^4_7\oplus\Lambda^4_{27},
\quad \Lambda^4_-=\Lambda^4_{35} $$ $$
\Lambda^5(M)=\Lambda^5_8\oplus\Lambda^5_{48} \quad \Lambda^6(M)=\Lambda^6_7\oplus\Lambda^6_{21},  \quad  \Lambda^7(M)=\Lambda^7_{8};$$

 \vspace{.1in}

\noindent where $\Lambda^4_{\pm}(M)$ are the $\pm$-eigenspaces of $*$ on $\Lambda^4(M)$ and
 
  \vspace{.1in}

 $$ \Lambda^2_7 = \{\alpha \in \Lambda^2(M) |
*(\alpha\wedge\Phi)=3\alpha\}, \quad \Lambda^2_{21} = \{\alpha \in
\Lambda^2(M)|*(\alpha\wedge\Phi)=-\alpha\} ,$$ $$ \Lambda^3_8 =
\{*(\beta\wedge\Phi) | \beta \in \Lambda^1(M)\}, \quad
\Lambda^3_{48} = \{\gamma \in \Lambda^3(M) | \gamma\wedge\Phi=0\},$$ $$\quad \Lambda^4_1 = \{f\Phi | f\in {\mathcal F(M)}\} $$ 

 \vspace{.1in}

The Hodge star
$*$ gives an isometry between $\Lambda^k_l$ and $\Lambda^{8-k}_l$.

 \vspace{.1in}


 \subsection{Almost symplectic structures and $Spin(7)$-structures}

 \vspace{.1in}
 
\noindent An almost symplectic manifold  $M$ is a n-dimensional manifold (n=2m) with a non degenerate 2-form $\omega$. If besides, $\omega$ is closed then $M$ is called a symplectic manifold. An almost-symplectic structure defines an $Sp(m, \mathbb{R})$ structure. A necessary and sufficient condition for the existence of an almost-symplectic structure on $M$ is the reduction of the structure group of the tangent bundle to the unitary group $U(m)$. It is, therefore, necessary that all odd-dimensional Stiefel-Whitney classes of $M$ vanish.
\vspace{.05in}


\noindent Next, we discussed a special class of $Spin(7)$ manifolds that admit an almost complex structure and show how it is related to the $Spin(7)$-structure. 

 \vspace{.1in}
 
Let $(M,\Phi)$ be a $Spin(7)$ manifold (or more generally manifold $Spin(7)$ structure) admitting a non-vanishing 2-plane field $\Lambda=\{u, v\} \in TM$. In \cite{thomas1969vector}, E. Thomas showed that the Euler characteristic $\chi(M)=0$ and the signature $\sigma(M)\equiv 0$ (mod 4) provides the necessary and sufficient conditions for the existence of a 2-plane field on an 8-manifold.

	\begin{definition}\label{def:cross}
		
	Let $(M,\Phi)$ be a $Spin(7)$ manifold. Then $J_{u,v} (z)= u \times v \times z$ is the triple cross product defined by the identity:
	\begin{equation}	< J_{u,v} (z),w > =  \Phi (u,v,z,w).
	\end{equation}
	where $\times$ is the octanionic cross product in $\mathbb{O}$.
\end{definition}

\begin{theorem}\cite{SE} Let $(M,\Phi)$ be a $Spin(7)$ manifold with a non-vanishing oriented 2-plane field. Then, $J_{u,v}(z)=u \times v \times z $ defines an almost complex structure on $M$ compatible with the $Spin(7)$ structure.

\end{theorem}

\begin{proof} 
	Let  $\{u, v\} \in TM$ be two vectors generating the non-vanishing oriented 2-plane field. $J(z)$ is well defined since by Equation (1), $<J(z),w>= \Phi (u,v,z,w).$

	\vspace{.1in}
	
	Next, we show $J^2(z)= - id$. This can be done using the properties of the $Spin(7)$-structure on $M.$  Let $z_i, z_j  \in TM$, 
	
	\vspace{.1in}

	\noindent Then 
	
	\vspace{.1in}	
	$<u \times v \times (u \times v \times \ z_i), z_j > = \Phi (u,v,(u\times v \times z_i), z_j) $
	
	\hspace{1.68in} $= -\Phi (u,v,z_j,(u\times v \times z_i)) $
	
	\hspace{1.68in} $= - <u \times v \times z_j, u \times v \times z_i> $
	
	\hspace{1.68in} $ =-\delta_{ij} $
	
	\vspace{.1in}	
	
	The last equality holds since the map $J$ is orthogonal.  Note that the map $J$ only depends on the oriented 2-plane $ \Lambda=\{u, v\}$.
	
\end{proof}

\vspace{0.04in}
As a result, Spin(7) manifolds inherently possess an almost complex structure as well as an almost symplectic structure, characterized by the presence of 2-plane fields. By means of this result, 4-frame on Spin(7)-manifold yields dual Calabi-Yau manifolds depending on position of 2-frames according to each other.



	
	
	
	
	
	
	


  \section{Mirror Duality on Exceptional Holonomy}


In the papers, \cite{akbulut2006mirror}, \cite{akbulutsalur} and \cite{gloversalur}, it is shown that the rich geometric structures of a $G_2$ manifold $N$ with $2$-plane fields provide complex and symplectic structures to certain 6-dimensional subbundles of $T(N)$. Using the 2-plane fields, the mathematical definition of "mirror symmetry" for Calabi-Yau and $G_2$ manifolds can be introduced. More specifically, one can assign a $G_2$ manifold $(N,\varphi, \Lambda)$, with the calibration 3-form $\varphi$ and an oriented $2$-plane field $\Lambda$, a pair of parametrized tangent bundle valued 2 and 3-forms of $N$. These
forms can then be used to define different complex and symplectic structures on certain 6-dimensional subbundles of $T(N)$. When
these bundles are integrated they give mirror CY manifolds. This is one way of explaining the duality between the symplectic and complex
structures on the CY 3-folds inside of a $G_2$ manifold.

As one might expect, understanding 4-vector fields on a Spin(7) manifold might help us to understand the properties of the Spin(7) metric and mirror duality.

\subsection{Cayley-free embeddings on Spin(7) manifold}
One can similarly construct these structures and define mirror dual Calabi Yau manifolds inside a $Spin(7)$ manifold which admits an almost complex structure.

\noindent Let $(M, \Phi)$ be a calibrated manifold. A p-plane $\xi$ is said to be tangential to a submanifold $X \subset M$ if
$span\xi \in T_xX$ for some $x \in X$. $\Phi$-free embedding is given by Harvey and Lawson \cite{harvey2009}. 

\begin{definition}A closed submanifold $X \subset M$ is called $\Phi$-free if there are no $\Phi$-planes $\xi \in G(\Phi)$ which are
tangential to $X$. On a Spin(7)-manifold with Cayley calibration $\Phi$, we call these submanifolds Cayley-free.
\end{definition}

\begin{definition} The free dimension of a calibrated manifold $(M, \Phi)$, denoted by $fd(\Phi)$, is the maximum
dimension of a linear subspace in $T_xM$ for $x \in M$ which contains no $\Phi$-planes. Subspaces with this property are called $\Phi$-free.
\end{definition}

\noindent As an example, if $(M, \omega)$ is a
K\"ahler manifold with real dimension $2n$, then  $fd(\omega) = n$. For a Spin(7)-manifold with Cayley calibration $\Phi$ Unal \cite{Unal}  showed the following result.

\begin{theorem} The free dimension of $\Phi$, $fd(\Phi)$ is equal to 4.
    
\end{theorem}

\noindent Since free dimension $fd(\Phi)=4,$ One can find non-vanishing 4-frame $V= \{u,v,y,w\}$ on Spin(7) manifold such that $\Phi\lvert_V=0.$  

 \subsection{Characteristic Classes of $G_4\mathbb{R}^8$}

Let $E = E(4, 8)$ and $F = F(4, 8)$ over $G^+_4 \mathbb{R}^8$ be the canonical vector bundles with
fibers generated by vectors of the subspaces and the vectors orthogonal to the subspaces, respectively. We have $$p_t(G^+_4 \mathbb{R}^8) = 1 + 3t^4 + 4t^8 + 3t^{12} + t^{16}.$$

The given polynomial yields the Betti numbers of $G^+_4 \mathbb{R}^8.$ The class of Cayley embeddings in $\mathbb{R}^8$ is a submanifold in Grassmannian manifold $G_4^+\mathbb{R}^8$ \cite{SZ}. 

\vspace{0.2in}

\begin{center}
    \includegraphics[height=70mm,width=110mm,]{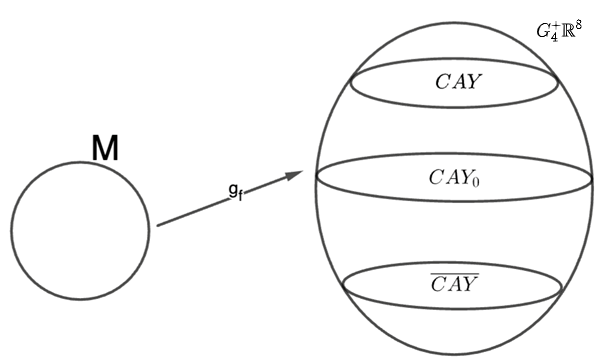}\captionof{figure}{}
    \label{fig:my_label}
\end{center}

\vspace{0.2in}

Let $f: M^4 \rightarrow \mathbb{R}^8$ be an embedding of closed oriented 4-manifold $M^4$ into $\mathbb{R}^8.$ Then the Gauss map of $f$ is written as $g_f: M^4 \rightarrow G_4^+(\mathbb{R}^8)$ shown in figure \ref{fig:my_label}.
\vspace{0.1in}

\noindent The tangent space of Spin(7) manifold is pointwisely isomorphic to  $\mathbb{R}^8.$
Let $CAY$,  $ CAY_0$  denote Cayley embeddings and Cayley-free embeddings respectively in Grassmannian $G_4\mathbb{R}^8$ where $g_f$ is the Gauss map. We can consider \( G_{f_0}(\mathbb{M}^4) \) as an embedded closed submanifold in \( G_4^+(\mathbb{R}^8) \) intersecting with $CAY$, and  $ CAY_0$ transversally in \( G_4^+(\mathbb{R}^8) \).

\begin{theorem} \cite{SZ} Let $E = E(4, 8)$ and $F = F(4, 8)$ over $G^+_4 \mathbb{R}^8$ be the canonical vector bundles with
fibers generated by vectors of the subspaces and the vectors orthogonal to the subspaces, respectively. Then
we have: \begin{enumerate}

\item $[G^+_4 \mathbb{R}^7]$, $[G^+_3 \mathbb{R}^7]$ and $[CAY]$ are generators of $H^{12}(G^+_4 \mathbb{R}^8; R)$
\item  $e(E), e(F)$ and $\frac{1}{2} (p_1(E) + e(E) - e(F))$ are generators of $H^4(G^+_4 \mathbb{R}^8; \mathbb{R})$ and their Poincaré duals are $[G^+_4\mathbb{R}^7], [G^+_3 \mathbb{R}^7]$, $[CAY] + [G^+_4 \mathbb{R}^7] - [G^+_3 \mathbb{R}^7] \in H_{12}(G^+_4 \mathbb{R}^8; \mathbb{R}),$ respectively.
\item  $[G^+_2 \mathbb{R}^4], [G^+_1 \mathbb{R}^5]$, and $[G^+_4 \mathbb{R}^5]$ are generators of $H_4(G^+_4 \mathbb{R}^8; \mathbb{R})$ and the following table shows the value of integration of certain characteristic classes on these generators.
\end{enumerate} 

\begin{table}[h]
    \centering
    \begin{tabular}{|c|c|c|c|}
        \hline
         & $G^+_2 \mathbb{R}^4$ & $G^+_1 \mathbb{R}^5$ & $G^+_4 \mathbb{R}^5$ \\
        \hline
        e(E) & 0 & 0 & 2 \\
        \hline
        e(F) & 0 & 2 & 0 \\
        \hline
        $p_1(E)$ & 2 & 0 & 0 \\
        \hline
    \end{tabular}
    \label{tab:my_table}
\end{table}
\end{theorem}

As a result of the theorem, we obtain that the Poincaré dual of $[CAY_0]$ is equal to $e(E)+e(F)$ since $[CAY_0]= [G^+_4 \mathbb{R}^7] + [G^+_3 \mathbb{R}^7].$

Since $f$ is an embedding into $\mathbb{R}^8,$ the following holds; 

$$ [g_f(M)]=\frac{1}{2}\chi_M [G^+_4 \mathbb{R}^5] + \frac{3}{2}\tau_M[G^+_2 \mathbb{R}^4].$$

\begin{lemma} The intersection number between classes $[CAY_0]$ and $ [g_f(M)]$ is 
$$ [g_f(M)] \cdot [CAY_0] = \chi_M $$
\end{lemma}

\begin{proof} . If $D_G : H_{12}(G_4 (\mathbb{R}^8); \mathbb{R}) \rightarrow H^4(G_4 (\mathbb{R}^8; \mathbb{R})$ is the Poincaré duality (or its inverse) map, then 
\vspace{0.2in}

\begin{tabular}{ccl}
     $[g_f(M)] \cdot [CAY_0]$ & = & $\bigints_{[g_f(M)]} D_G(CAY_0)$ \\
    & = &  $\bigints_{[g_f(M)]} (e(E) + e(F))$ \\
    & = &   $\frac{1}{2}\chi_M\bigints_{[G^+_4 \mathbb{R}^5]} e(E) + \frac{3}{2}\tau_M\bigints_{[G^+_2 \mathbb{R}^4]}  e(E) $ \\
    & = & $\chi_M$
\end{tabular}

\end{proof}

\noindent Cayley-free embedding is obtained by the Euler characteristic of the submanifold $M.$ Unal showed the following;
\begin{lemma} \cite{Unal} If $\chi_M=\tau_M=0$, then there exists a homotopy $F:[0,1]\times M \rightarrow G_4 (\mathbb{R}^8)$ such that $F_0(M) =
g_f(M)$ and $F_1(M) \cap CAY = \emptyset,\ \ F_1(M) \cap \overline{CAY} = \emptyset$.

Hence, one can choose a nonvanishing 4-frame  with some special topological conditions such that 4-frame is Cayley-free. 

\end{lemma} 
\begin{proposition}

 Cayley submanifold of Spin(7) manifold $M$ is unique under Hamiltonian isotopy.

\end{proposition}

\begin{proof}
The Maslov index of Cayley embedding is induced from the first cohomology class. Also, we know that 
$$CAY\cong G_3\mathbb{R}^7$$ 
Akbulut and Kalafat  \cite{AK} stated that $H^1(G_3\mathbb{R}^7;\mathbb{Z})=0$. In conclusion, the Maslov index is zero for Cayley submanifolds. In other words, under Hamiltonian isotopy, the Cayley submanifold is unique. 

\end{proof}
\noindent Hence Hamiltonian isotopy preserves the topology of symplectic manifolds, the fact that 
$M$ is unique in its isotopy class suggesting that certain topological features of $M$ are also preserved within this class.

\noindent In conclusion,  the Floer homology of Cayley submanifolds isomorphic to De Rham homology. For Cayley submanifolds which are non-singular can be induced by torus fibration.

 \subsection{4-frame on Spin(7) manifold} Under some topological restriction, a Spin(7) manifold admits a 4-frame field \cite{cadek}. A 4-frame $\xi$ is called \textit{Cayley}, if it is closed under a three-fold cross product defined at the definition (\ref{def:cross}) for the basis of 4-plane $\{ u\times v \times w, u, v, w\}.$

\noindent \u{C}adek and Van\u{z}ura \cite{cadek} showed the sufficient condition for the existence of a 4-plane field on Spin(7) manifold.

\begin{lemma}\cite{cadek} Let $M$ be Spin(7)-manifold with sixth Stiefel-Whitney class $w_6=0$, Then there exist 4 linearly independent vector fields on $M.$
\end{lemma}

\vspace{.1in}

 \begin{proposition} \cite{OP} Let $\xi$ be a 4-plane in $M^8.$ $\xi$
is Cayley if and only if 
$$x(\bar{y}z)\cdot u=z(\bar{y}x)\cdot u$$
     \end{proposition}

 where $\bar{x}$ is conjugate of $x$ and $\cdot$ is multiplication in Octanion.

\begin{lemma}Let $M$ be a Spin(7) manifold with $w_6=0$ and $\{u,v,y,w\}$ span $X$ as a 4-vector field. Then, two almost complex structures defined by triple product define mirror duality. 
$J_{uv}$ and $J_{yw}$ are the almost complex structures on 8-manifold $M.$  Then, $K=J_{uv}J_{yw}$ is the other almost complex structure.\end{lemma}

\begin{proof}
    Let $M$ be a $\text{Spin}(7)$ manifold with $w_6 = 0$. Let $\{u, v, y, w\}$ span $X$ as a 4-vector field. We define two almost complex structures $J_{uv}$ and $J_{yw}$ on the 8-manifold $M$. We aim to prove that $K = J_{uv}J_{yw}$ is another almost complex structure on $M$ and that these structures define a mirror duality.
\noindent Define the almost complex structures $J_{uv}$ and $J_{yw}$ such that:
   \[
   J_{uv}(u) = v, \quad J_{uv}(v) = -u, \quad J_{uv}(y) = w, \quad J_{uv}(w) = -y
   \]
   and similarly for $J_{yw}$:
   \[
   J_{yw}(y) = w, \quad J_{yw}(w) = -y, \quad J_{yw}(u) = v, \quad J_{yw}(v) = -u.
   \]

\noindent By construction, $J_{uv}^2 = J_{yw}^2 = -\text{id}$, verifying that both $J_{uv}$ and $J_{yw}$ are almost complex structures. These maps satisfy the defining property of an almost complex structure: $J^2 = -\text{id}$.

\noindent Consider the map $K = J_{uv}J_{yw}$. We need to verify that $K$ is also an almost complex structure. Compute $K^2$:
   \[
   K^2 = (J_{uv} J_{yw})(J_{uv} J_{yw}) = J_{uv} (J_{yw} J_{uv}) J_{yw} = J_{uv} (-\text{id}) J_{yw} = -J_{uv} J_{yw}.
   \]

 \noindent  To show that $K$ is an almost complex structure, we need $K^2 = -\text{id}$:
   \[
   K^2 = -J_{uv} J_{yw} = -\text{id}.
   \]

  \noindent Therefore, $K$ is indeed an almost complex structure.

\noindent In the context of mirror symmetry in string theory, mirror duality typically involves a pair of manifolds whose geometric structures (complex, symplectic, etc.) are interchanged. Here, the almost complex structures $J_{uv}$ and $J_{yw}$ can be seen as defining a duality on the same manifold $M$ through their interaction. The map $K$ arising from their product represents a new structure that is intrinsically linked to both $J_{uv}$ and $J_{yw}$, thus reflecting a form of duality between these structures.

\end{proof}

\noindent Once mutual intersection 2-frames are line, this yields hypercomplex structure as follow;
\begin{proposition}\cite{OP} A 4-plane $\xi$ in $\mathbb{R}^8$
is Cayley if and only if any triple of
mutually orthogonal 2-planes $\alpha, \beta, \gamma \subset \xi$, all intersecting in a line, defines a hypercomplex structure on $\xi$.

     \end{proposition}




 \subsection{Local Coordinates on Spin(7) manifold}  
In this section, Spin(7) manifold will be analyzed locally with the decomposition at section \ref{decom}.

The following proposition is given by Karigiannis \cite{karigiannis} to state local behavior of 4-frame.
\begin{proposition} \cite{karigiannis} \label{propa} Let $u$, $v$, $y$, and $w$ be vector fields. 
Furthermore, if we define
\begin{eqnarray} 
A & = & \langle u \wedge v, y \wedge w \rangle = \langle u, y
\rangle \langle v, w \rangle - \langle u, w \rangle \langle v, y
\rangle \\ B & = & \Phi (u,v,y,w) 
\end{eqnarray}
then the following relations hold between these
$2$-forms:
\begin{eqnarray}
 (\iota_u\iota_\v \Phi) \wedge (y^\flat \wedge w^\flat)
\wedge \Phi & = & \left(-3 A -2 B \right) vol \\
(u^\flat \wedge v^\flat) \wedge (\iota_y \Phi)
\wedge (\iota_w \Phi) & = & \left(-4 A + 2 B \right) vol \\
\label{propi} (\iota_u \iota_v\Phi) \wedge (\iota_y \iota_w\Phi)
\wedge \Phi & = & \left( 6 A + 7 B \right) vol 
\end{eqnarray}
where $\iota_{(.)} \Phi$ is contraction of the Cayley form  $\Phi$ and $(.)^\flat$ is the 1-form of the given vector field.
\end{proposition}

\vspace{0.1in}

\begin{lemma} \label{propb} Let  \(\{u, v, y, w\}\) be Cayley free 4-frame on Spin(7) manifold.

\begin{eqnarray} 
A & = & \langle u \wedge v, y \wedge w \rangle = \langle u, y
\rangle \langle v, w \rangle - \langle u, w \rangle \langle v, y
\rangle \\ B & = & \Phi (u,v,y,w) 
\end{eqnarray}
Then,

\begin{eqnarray}
 (\iota_u\iota_\v \Phi) \wedge (y^\flat \wedge w^\flat)
\wedge \Phi & = & (-3 A) vol \\
\label{threeform} (u^\flat \wedge v^\flat) \wedge (\iota_y \Phi)
\wedge (\iota_w \Phi) & = & (-4 A ) vol \\
 (\iota_u \iota_v\Phi) \wedge (\iota_y \iota_w\Phi)
\wedge \Phi & = & ( 6 A ) vol 
\end{eqnarray}
where $\iota_{(.)} \Phi$ is contraction of the Cayley form  $\Phi$ and $(.)^\flat$ is the 1-form of the given vector field.
\end{lemma}
\begin{proof}
    
$\Phi(u,v,y,w)=0$ for Cayley-free 4-frame, the lemma is obtained from proposition \ref{propa}.

\end{proof}

\noindent Hence, this yields mirror dual Calabi-Yau manifolds by splitting 4-frame into two 2-frames as follows; 

   \begin{proposition} Let $M$ be a Spin(7) manifold with Cayley-free 4-plane $\{u,v,y,w\}$. 
  Then, 
  non-degenerate two forms $\omega_{uv}$ and $ \omega_{yw}$  generates Calabi-Yau manifolds $(N_1,J_{uv},\omega_{uv},Re\ \Omega_{uv},Im\ \Omega_{uv})$ and $(N_2,J_{yw},\omega_{yw},Re\ \Omega_{yw},Im\ \Omega_{yw})$ mirror dual of each other induced from $\Phi$.
   \end{proposition}
   \vspace{0.2in}
   
   \begin{proof}


  Let \( M \) be an 8-dimensional manifold with \(\text{Spin}(7)\) holonomy. Consider a Cayley-free 4-plane spanned by \(\{u, v, y, w\}\).

 \noindent  Define the two-forms \(\omega_{uv}\) and \(\omega_{yw}\) by:
   \[
   \omega_{uv} = u^\flat \wedge v^\flat, \quad \omega_{yw} = y^\flat \wedge w^\flat,
   \]
   where \( u^\flat, v^\flat, y^\flat, \) and \( w^\flat \) are the 1-forms dual to the vector fields \( u, v, y, \) and \( w \), respectively.

  \noindent The forms \(\omega_{uv}\) and \(\omega_{yw}\) are non-degenerate as they are defined on a Cayley-free 4-plane in a \(\text{Spin}(7)\) manifold, which ensures that the plane is not associated with any singularities or degeneracies.

   \noindent Each 2-form \(\omega_{uv}\) and \(\omega_{yw}\) defines a symplectic structure on a 6-dimensional submanifold \(N_1\) and \(N_2\) respectively:
   \[
   \omega_{uv} = \frac{1}{2} \sum_{i,j} \omega_{uv,ij}\ dx^i \wedge dx^j,
   \]
   \[
   \omega_{yw} = \frac{1}{2} \sum_{i,j} \omega_{yw,ij}\ dx^i \wedge dx^j.
   \]

  \noindent The almost complex structures \(J_{uv}\) and \(J_{yw}\) on \(N_1\) and \(N_2\) are induced by the action of \(\text{Spin}(7)\) on the tangent spaces:

    \begin{center}

 \begin{tabular}{l l l}
    $J_{uv}(z)=u\times v \times z$  & and & $J_{yw}(z)=y \times w \times z$ \\

 \end{tabular}
 
  \end{center}
   \[
   J_{uv}(u) = v, \quad J_{uv}(v) = -u, \quad J_{uv}(y) = w, \quad J_{uv}(w) = -y,
   \]
   \[
   J_{yw}(y) = w, \quad J_{yw}(w) = -y, \quad J_{yw}(u) = v, \quad J_{yw}(v) = -u.
   \]

  \noindent The Calabi-Yau structures on \(N_1\) and \(N_2\) are completed by holomorphic volume forms:
   \[
   \Omega_{uv} = dz_1 \wedge dz_2 \wedge dz_3, \quad \Omega_{yw} = dw_1 \wedge dw_2 \wedge dw_3,
   \]
   where \(dz_i\) and \(dw_i\) are the complex coordinates compatible with the almost complex structures \(J_{uv}\) and \(J_{yw}\), respectively. Moreover, volume forms are compatible to each other according to the lemma \ref{propb} at (\ref{threeform}).

   \vspace{0.07in}

   \noindent The mirror dual Calabi-Yau manifolds $(N_1,J_{uv},\omega_{uv},Re\ \Omega_{uv},Im\ \Omega_{uv})$ \linebreak and $(N_2,J_{yw},\omega_{yw},Re\ \Omega_{yw},Im\ \Omega_{yw})$ is induced by the \(\text{Spin}(7)\) structure. The structure \(\Phi\) on the \(\text{Spin}(7)\) manifold interchanges the roles of \(\omega_{uv}\) and \(\omega_{yw}\), and their corresponding complex structures and holomorphic volume forms, establishing the mirror duality. Hence, Calabi-Yau manifolds are mirror duals of each other in $M$. 
   \end{proof}
   
\noindent This generalizes the idea given by Akbulut and Salur at \cite{akbulut2006mirror}. They constructed mirror dual Calabi-Yau submanifolds induced by 2-vector field on $G_2$-manifold in an 8-manifold. In the same way, we construct mirror dual  Calabi-Yau manifolds induced from the same calibration $\Phi$ with 4 vector fields. In this result, one can construct mirror dual Calabi-Yau manifolds in a Spin(7) manifolds with Cayley-free 4-frame.

 \subsection{Example}

 Let $M$ be a $\mathbb{R}^4$ bundle over $S^4.$ Gibbon, Page and Pope \cite{GPP}  showed that there exists a covariantly-constant spinor for Ricci-flat metric on the  $\mathbb{R}^4$ bundle over $S^4.$ The irreducible spinor representations of the  Spin(8)  double covering of the $SO(8)$ tangent group are the $8_+$ and $8_-,$ which corresponds to to left-handed and right-handed Majorana-Weyl spinors respectively. Under embedding of Spin(7) in Spin(8), one of the spinors decomposes irreducibly, say $8_+\rightarrow 8_-,$ while the other decomposes at $8_+\rightarrow 7+1.$ By this decomposition, we see a hypercomplex structure on Spin(7) manifold as follows $J_{e_1e_2}(e_5)=-e_3$, $J_{e_3e_5}(e_1)=e_2$, $J_{e_1e_2}(e_3)=-e_5$ such that $J_{e_1e_2}J_{e_3e_5}=-id$ where  Cayley-free  4-frame $\{e_1,e_2,e_3,e_5\}$ is standard basis of $\mathbb{R}^4.$
 
\vspace{0.07in} 
\noindent Future research will focus on some compact examples and some properties of mirror dual Calabi-Yau manifolds induced from calibration $\Phi$ with 4 vector fields with singularities.

\vspace{0.07in}
\noindent \textbf{Acknowledgement} The author would like to thank the Scientific and Technological Research Council of Turkey (T\"UB\.ITAK) and Sema Salur for helpful advice on this issue.

\end{document}